\documentclass{amsart}

\usepackage{amsmath, amsthm,amssymb,latexsym,graphicx,xcolor,pstricks,setspace,enumerate}

\newtheorem{theorem}{Theorem}[section]
\newtheorem{lemma}{Lemma}[section]
\newtheorem{corollary}{Corollary}[section]
\newtheorem{proposition}{Proposition}[section]

\theoremstyle{definition}
\newtheorem{definition}[theorem]{Definition}

\newtheorem{assumption}[theorem]{Assumption}
\newcommand{\lspan}{{\rm span \,}}

\theoremstyle{remark}
\newtheorem{remark}[theorem]{Remark}

\numberwithin{equation}{section}

\newcommand{\cin}{C^\infty}
\newcommand{\what}{\ensuremath{\widehat}}

\newcommand{\R}{\ensuremath{\mathbb{R}}}

\newcommand{\norm}[1]{\|#1 \|}
\newcommand{\norml}[1]{\left\|#1 \right\|}
\newcommand{\Hc}{{\mathcal{H}}}
\newcommand{\Lc}{{\mathcal{L}}}
\newcommand{\be}[1]{\begin{equation}\label{#1}}
\newcommand{\ee}{\end{equation}}
\newcommand{\ino}{{i_1}}
\newcommand{\ind}{{i_2}}
\newcommand{\st}{\; : \;}
\newcommand{\rref}[1]{(\ref{#1})}

\newcommand{\Uc}{{\mathcal  U}}
\newcommand{\Oc}{{\mathcal  O}}
\newcommand{\Rc}{{\mathcal R}}
\newcommand{\bd}[1]{\begin{definition}\label{#1}}
\newcommand{\ed}{\end{definition}}
\newcommand{\bt}[1]{\begin{theorem}\label{#1}}
\newcommand{\et}{\end{theorem}}
\newcommand{\bc}[1]{\begin{corollary}\label{#1}}
\newcommand{\ec}{\end{corollary}}

\newcommand{\Rb}{{\mathbb R}}  
\newcommand{\ip}[2]{\langle #1 , #2 \rangle}
\newcommand{\bp}[1]{\begin{proposition}\label{#1}}
\newcommand{\ep}{\end{proposition}}
\newcommand{\bpr}{{\em Proof. }}
\newcommand{\epr}{\hfill $\square$}
\newcommand{\comment}[1]{{ }}
\newcommand{\whp}[1]{\what{P}_{#1}}
\newcommand{\whq}[1]{\what{Q}_{#1}}
\newcommand{\wha}[1]{\what{A}_{#1}}

\newcommand{\whr}[1]{\what{R}_{#1}}
\newcommand{\whv}[1]{\what{V}_{#1}}

\newcommand{\dop}[1]{\frac{d}{d #1}}

\newcommand{\ad}{\mathrm{ad}\,}

\doublespace
\begin{document}

\title[Extension of Chronological Calculus]{Extension of Chronological Calculus for Dynamical Systems on Manifolds}

\author[Kipka]{Robert J. Kipka}
\address[Kipka]{Department of Mathematics, Western Michigan University, Kalamazoo, MI, 49008-5248}
\email{robert.j.kipka@wmich.edu}

\author[Ledyaev]{Yuri S. Ledyaev}
\address[Ledyaev]{Department of Mathematics, Western Michigan University, Kalamazoo, MI, 49008-5248}
\email{iouri.lediaev@wmich.edu}

\begin{abstract}
We propose an extension of the Chronological Calculus, developed by Agrachev and Gamkrelidze for the case of $\cin$-smooth dynamical systems on finite-dimensional $\cin$-smooth
manifolds, to the case of $C^m$-smooth dynamical systems and infinite-dimensional
$C^m$-manifolds. Due to a relaxation in the underlying structure of the calculus, this extension provides a powerful computational tool without recourse to the
theory of calculus in Fr\'echet spaces required by the classical Chronological Calculus. In addition, this extension accounts for flows of vector fields which are merely measurable in time. To demonstrate the utility of this extension, we prove a variant of Chow-Rashevskii theorem for  infinite-dimensional manifolds.
\end{abstract}



\maketitle

\tableofcontents

\section{Introduction and Background}

In the 1970s, Agrachev and Gamkrelidze suggested in \cite{MR524203,MR579930} the Chronological Calculus for the analysis of
$C^\infty$-smooth dynamical systems on
finite-dimensional manifolds (for a textbook exposition see \cite{MR2062547}). The central idea of this calculus is to consider flows of
dynamical systems as linear operators on the space of $C^\infty$-smooth scalar functions. Such ``linearization'' of
flows on manifolds presents significant advantages from the point of view of defining derivatives of flows, developing
a calculus of such derivatives, and effective computations of formal power series representing flows.

But in addition to these desirable properties, the Chronological Calculus poses some interesting problems. The space of $C^\infty$-smooth scalar functions is a Fr\'echet space with topology given by a countable family of seminorms and this
complicates the proofs of the calculus rules given in \cite{MR579930,MR2062547}. The approach also
requires the strong assumption of $C^\infty$-smoothness of dynamical systems and manifolds, even if for many applications
only finite sums of Volterra-like series representing flows are enough \cite{MR2374261}.

Another restriction of the classical Chronological calculus (which is important from the point of view of applications to control systems on manifolds) is its treatment  of nonautonomous vector fields which depend on $t$ in measurable way. In particular, there is no variant of product rule in the classical Chronological Calculus which can be used for such flows.

In this paper we extend the Chronological Calculus so as to require only $C^m$-smoothness of dynamical
systems and manifolds. The result is a computationally effective version of the Chronological Calculus without recourse to Fr\'echet spaces. Moreover, in the framework of this extension we provide a ``distributional'' version of product rule which
can be applied for nonautonomous flows with only measurable in $t$  vector fields. Thus we give details for a rule which are lacking in the description of the classical Chronological Calculus \cite{MR524203,MR579930,MR2062547,MR1012328}, even for finite-dimensional manifolds.

 Further, this extension allows analysis of dynamical systems on infinite-dimensional manifolds, which are interesting from the point of view of applications to the theory of partial
differential equations. We also develop a calculus of remainder terms (calculus of ``little $o$'s'') which is used for the effective calculation of representations of  brackets of  flows in terms of respective brackets of vector fields on infinite-dimensional
manifolds and which provides an algorithm for the computation of remainder terms in such representations. Finally, we use these results for proving a generalization
of Chow-Rashevskii theorem for infinite-dimensional manifolds.

In order to take a comprehensive approach to the problem, we begin by recalling some facts of calculus and differential geometry in Banach spaces.

\subsection{Calculus in a Banach Space}

Let $E$ and $F$ be Banach spaces. A map $f : E \rightarrow F$ is said to be \emph{differentiable} at $x_0$ if there exists a bounded linear operator
$f^\prime(x_0) : E \rightarrow F$ such that for all $x \in E$ we have
$f(x) = f(x_0) + f^\prime(x_0)\left(x - x_0\right) + o \left( \left\|x - x_0\right\| \right)$.
If $f$ is differentiable on all of $E$, then we have $f^\prime : E \rightarrow L(E,F)$, where $L(E,F)$ is the Banach space of bounded linear operators
from $E$ to $F$. When $f^\prime$ is continuous, we say that $f$ is of class $C^1$. As a map between Banach spaces, we may then ask if $f^\prime$ is differentiable and so on.
If $f$ has $m$ continuous derivatives, then we say that $f$ is of class $C^m$. The $m^{th}$ derivative at a point $x_0$ may be identified with an $m$-multilinear map
$\underbrace{E \times \dots \times E}_{m\; \mathrm{copies}} \rightarrow F$ and the space of such maps is again a Banach space with norm
 \begin{equation*}
   \left\|A\right\|  =\sup \left\{ \left\|A(x_1,\dots, x_m)\right\| \, : \, \left\|x_1\right\| = \dots = \left\|x_m\right\| = 1 \right\}.
 \end{equation*}

Functions which take values in a Banach space can also be integrated. For a rigorous introduction to the integration of vector-valued functions, we recommend   \cite{diestel-uhl1977}. We briefly describe here the \emph{Bochner integral} for functions $f : \left[t_0,t_1\right] \rightarrow E$, where $E$ is a Banach space. As one might expect, a function $f : \left[t_0, t_1\right] \rightarrow E$ is said to be \emph{simple} if it takes on only finitely many values, say $\left[t_0,t_1\right] = \cup_{i = 1}^k A_i$, with $A_i$ disjoint measurable sets and $\left. f \right|_{A_i} = f_i \in E$. For simple functions one then defines
\begin{equation*}
  \int_{t_0}^{t_1} f \, dt = \sum_{i = 1}^k f_i \mu(A_i),
\end{equation*}
where $\mu$ is Lebesgue measure. If $E$ is a Banach space, a function $f : \left[t_0, t_1\right] \rightarrow  E$ is said to be \emph{measurable} if it is a pointwise limit of a sequence  of simple functions, say $f_n \to f$.   Measurable function $f$ is said to be \emph{Bochner integrable} if
 $\displaystyle \lim_n \int_{t_0}^{t_1} \norm{f-f_n}dt =0$ for some sequence of simple functions $f_n$ . In this case the \emph{Bochner integral} of $f$ is defined as
\[
  \int_{t_0}^{t_1} f(t) \, dt = \lim_{n} \int_{t_0}^{t_1} f_n(t) \, dt.
\]
  It is worth noting that when $E = \R^n$, the Bochner integral is the same as the Lebesgue integral. In general Banach spaces, the Bochner integral retains many desirable properties of the Lebesgue integral. In particular, one has
\begin{equation}
\label{eq:lebesgue}
  \frac{d}{dt} \int_{t_0}^t f(\tau) \, d \tau = f(t)
\end{equation}
for almost all $t$ in $\left[t_0,t_1\right]$. Function $F(t)$ is called \emph{absolutely continuous} if  $F(t)=F(t_0)+\int_{t_0}^t f(\tau) \, d \tau$ for some integrable $f$. This and other properties of Bochner integral are given a clear treatment in \cite{diestel-uhl1977}.

\subsection{Differential Equations and Flows in Banach Space}

We recall some results from the theory of differential equations in Banach spaces. In particular, we are interested in equations of the form
\begin{equation}
  \label{eq:banach-ivp}
  \dot{x} = f(t,x) \hspace{2cm} x(t_0) = x_0
\end{equation}
where $f : J \times E \rightarrow E$ and $J \subseteq \R$ is an interval containing $t_0$. An excellent resource for this material is \cite{MR0463601}. There it is demonstrated that in a Banach space, continuity of $f$ is not enough to ensure a solution. We introduce the following definitions for vector fields on $E$:
\begin{definition}
  A \emph{nonautonomous $C^m$ vector field} on $E$ is a function $f : J \times E \rightarrow E$ which is measurable in $t$ for each fixed $x$ and
$C^m$ in $x$ for almost all $t$.
\end{definition}

\begin{definition}
\label{defn:locally-integrably-bounded}
A nonautonomous $C^m$ vector field on $E$ is said to be \emph{locally integrable bounded} if for any $x_0 \in E$, there exists an open neighborhood $U$ of $x_0$ and $k \in L^1(J,\R)$ such that for all $x \in U$, for all $0 \le i \le m$, we have $\left\|f^{(i)}(t,x) \right\| \le k(t)$ for almost all $t$,
  where $f^{(i)}$ denotes the $i^{th}$ derivative of $f$ with respect to $x$.
\end{definition}

\begin{definition}
\label{defn:locally-bounded}
A nonautonomous $C^m$ vector field on $E$ is said to be \emph{locally bounded} if for any $x_0 \in E$, there exists an open neighborhood $U$ of $x_0$ and $K  \ge 0$ such that for all $x \in U$, for all $0 \le i \le m$, we have $\left\|f^{(i)}(t,x) \right\| \le K$ for almost all $t$.
\end{definition}

Notice that any autonomous $C^m$ vector field is locally bounded. It can be shown that if $f : J \times E \rightarrow E$ is a nonautonomous $C^m$ vector field that is locally integrable bounded, then for any $(t_0,x_0)$ there exists an open interval $J_0\subset J$ containing $t_0$ and depending on $(t_0,x_0)$ as well as a unique, absolutely continuous function $x : J_0 \rightarrow E$ which satisfies \eqref{eq:banach-ivp} for almost all $t \in J_0$. This type of solution is called a \emph{Carath\'eodory} solution. In addition, the dependence of this solution upon the initial condition $x_0$ is $C^m$-smooth. More precisely, if $x(t;t_0,x_0)$ denotes the solution to \eqref{eq:banach-ivp}, then $x_0 \mapsto x(t;t_0,x_0)$ is $m$ times continuously differentiable for appropriate values of $t$ and $x_0$.

We will write $P_{t_0,t}$ for the \emph{local flow} $x_0 \mapsto x(t;t_0,x_0)$.
Uniqueness of solutions gives us the following properties for the flow:
\begin{align}\label{sgp}
  P_{s,t} \circ P_{t_0,s}(x) & = P_{t_0,t}(x) \\
  P_{t_0,t}^{-1}(x) & = P_{t,t_0}(x)
\end{align}
When the underlying vector field is autonomous, we will write $P_t$ for $P_{0,t}$. One may then obtain the following local semigroup properties for the flow:
\begin{align*}
  P_s \circ P_t (x) & = P_{s+t}(x) \\
  P_t^{-1}(x) & = P_{-t}(x),
\end{align*}
provided that $t,s,t+s$, and $-t$ lie in $J_0$, an interval which in general will depend on $x$.

\subsection{Smooth Manifolds}

In defining dynamical systems, it is enough for the underlying space to have the structure of a Banach space only locally. In this section we remind the reader of some definitions and basic results from the theory of smooth manifolds.  For a greater level of detail, we suggest \cite{MR1666820}.

A \emph{Banach manifold} of class $C^m$ over a Banach space $E$ is a paracompact Hausdorff space $M$ along with a collection of coordinate charts
$\left\{\left(U_\alpha, \varphi_\alpha\right) \, : \, \alpha \in A\right\}$, where $A$ is an indexing set. This collection of charts should be such that
the collection $\left\{U_\alpha\right\}$ is a cover for $M$; each $\varphi_\alpha$ is a bijection of $U_\alpha$ with an open subset of $E$; and
 the transition maps $\varphi_\alpha \circ \varphi_\beta^{-1} : \varphi_\beta(U_\alpha \cap U_\beta) \rightarrow \varphi_\alpha(U_\alpha \cap U_\beta)$
  are of class $C^m$.

If $M$ and $N$ are Banach manifolds, a function $f : M \rightarrow N$ is said to be \emph{$C^m$-smooth} (or $C^m$ for brevity) if for any coordinate charts $\varphi : U \subseteq M \rightarrow E$ and $\psi : V \subseteq N \rightarrow F$ the map $\psi \circ f \circ \varphi^{-1}$ is a $C^m$-smooth mapping of Banach spaces. Analogously, a function $f : M \rightarrow N$ is \emph{differentiable} at a point $q_0$ if $\psi \circ f \circ \varphi^{-1}$ is differentiable at $\varphi(q_0)$.

The \emph{tangent space} to $M$ at $q$ is defined  as follows. Consider the collection
of differentiable curves $\gamma : \R \rightarrow M$ with $\gamma(0) = q$ and define an equivalence relation on this collection by $\gamma_1 \sim \gamma_2$ if
and only if $\left( \varphi \circ \gamma_1\right)^\prime(0) =\left( \varphi \circ \gamma_2\right)^\prime(0) $ for some coordinate chart $\varphi$. One can check that if this relationship holds for one coordinate chart, it will hold for all coordinate charts. We write $\left[\gamma\right]$ for the equivalence class of a curve $\gamma$. The collection of these equivalence classes forms the tangent space $T_qM$ and there is a natural isomorphism $T_qM \leftrightarrow E$.

Every $C^m$  map $f : M \rightarrow N$ induces a map from $T_qM$ to $T_{f(q)}N$ by $\left[\gamma\right] \mapsto \left[f \circ \gamma \right]$ and we denote this mapping by $f_*(q)$. The tangent bundle $TM$ is the union of the tangent spaces with a topology given locally by the charts $(q,v) \mapsto \left(\varphi(q), \varphi_*(q)v\right)$, where $\varphi$ is a coordinate chart for $M$. When $f$ is a map between linear spaces $E$ and $F$ we will write $f^\prime$ for its derivative. When $f$ is a map between Banach manifolds, we will write $f_*$ for the corresponding map from $TM$ to $TN$. We emphasize that in local coordinates, $f_*(q) : T_qM \rightarrow T_{f(q)}N$ is the map given by $v \mapsto f^\prime(q)v$. In contrast, the map $f_* : TM \rightarrow TN$ sends a pair $(q,v)$ to the pair $(f(q), f_*(q)v)$.

\subsection{Vector Fields and Flows on Manifolds}

Let $\pi : TM \rightarrow M$ be the projection $(q,v) \mapsto q$. A \emph{nonautonomous vector field} is a mapping
$V : \R \times M \rightarrow TM$ which satisfies $\pi \circ V_t(q) = q$. Given $q_0 \in M$ and a coordinate chart $\left(\varphi, U\right)$ at $q_0$, the function $J \times E \rightarrow E$ given by
\begin{equation}
\label{eq:local-representation}
  (\varphi_* V_t)(x) := \varphi_* \left(\varphi^{-1}(x) \right) V_t\left(\varphi^{-1}(x) \right)
\end{equation}
is the \emph{local coordinate representation for $V_t$}. Recalling definition \ref{defn:locally-integrably-bounded} we introduce

\begin{definition} \label{defn:locally-integrably-bounded-manifold} A nonautonomous vector field on $M$ is said to be a \emph{locally integrable bounded $C^k$  vector field} if it is $C^k$-smooth in $q$ for almost all $t$, is measurable in $t$, and in some neighborhood of each $q \in M$ there is a coordinate representation \rref{eq:local-representation} which is locally integrable bounded.
\end{definition}

Similarly, recalling definition \ref{defn:locally-bounded}, we introduce

\begin{definition} \label{defn:locally-bounded-manifold} A nonautonomous vector field on $M$ is said to be a \emph{ locally bounded $C^k$ vector field} if it is $C^k$-smooth in $q$ for almost all $t$, is measurable in $t$, and in some neighborhood of each $q \in M$ there is a coordinate representation \rref{eq:local-representation} which is locally bounded.
\end{definition}

If $x(t)$ is a solution for the differential equation $\dot{x} = (\varphi_* V_t)(x)$ on $E$ with initial condition $x(t_0) = \varphi(q_0)$, then $q(t) = \varphi^{-1} \circ x(t)$ is a solution to the differential
equation on $M$
\be{mde}
\dot{q} = V_t(q), \quad q(t_0)=q_0.
\ee
For any $\varphi\in C^m(M,E)$ we have the following integral representation
\be{rie}
\varphi(q(t))=\varphi (q_0)+\int_{t_0}^t \varphi_* (q(\tau))V_\tau(q(\tau))\, d\tau.
\ee
 With each  nonautonomous  vector field $V_t$ on $M$, we associate a local flow $P_{t_0,t}$ given by $q_0 \mapsto q(t;t_0,q_0)$, the solution to \rref{mde}
with initial condition $q(t_0) = q_0$. In the case of autonomous vector fields $V$ we consider a local flow $P_t:q_0\mapsto q(t;0,q_0)$. These flows are $C^m$ diffeomorphisms of $M$ and  are of central importance in the development of our extension of the Chronological Calculus, which we now turn to.

\section{Extension of Chronological Calculus}

The main observation behind the Chronological Calculus \cite{MR524203,MR579930,MR2062547} is that one may trade analytic objects such as diffeomorphisms and vector fields for
algebraic objects such as automorphisms and derivations of the algebra $C^\infty(M)$, which is the collection of $C^\infty$  mappings $f : M \rightarrow \R$. This correspondence is developed in \cite{MR524203,MR579930,MR2062547},
where $C^\infty(M)$ is given the structure of a Fr\'echet space. Below
we develop a streamlined version of the theory which is effective for computations with infinite-dimensional $C^m$-manifolds and dynamical systems. In order
to include Banach spaces in the theory, we consider the vector space $C^m(M,E)$ of $C^m$ functions $f: M \rightarrow E$ rather than the algebra of scalar functions  $C^\infty(M)$.

\subsection{Chronological Calculus Formalism: flows as linear operators}

We begin by defining the following operators:
\begin{enumerate}
\item[i.] The \emph{identity} operator $\what{Id}_M$ is defined as follows
$\what{Id}_M(\varphi)=\varphi$ for any $\varphi\in C(M,E)$.
  \item[ii.] Given any point $q \in M$, let $\what{q} : C^m(M,E) \rightarrow E$ be the linear map given by $\what{q}(\varphi) := \varphi(q)$.
  \item[iii.] Given $C^m$-manifolds $M$ and $N$ over a Banach space $E$ and a $C^m$  map $P : M \rightarrow N$, let
  $\what{P} : C^m(N,E) \rightarrow C^r(M,E)$ ($0 \le r \le m$) be the linear map given by $\what{P}(\varphi) := \varphi \circ P$. Note that if $P$ is a diffeomorphism of $M$, $\what{P}$
  gives us an isomorphism of $C^m(M,E)$.
  \item[iv.] Given a tangent vector $v \in T_qM$, let $\what{v} : C^m(M,E) \rightarrow E$ be the linear map given by $\what{v}(\varphi) := \varphi_*(q)v$.
  \item[v.]   Given any $C^m$ vector field $V$ on $M$, we define a linear map $\what{V} : C^m(M,E) \rightarrow C^{m-1}(M,E)$  by
  $\what{V}(\varphi) : q \mapsto \varphi_*(q)V(q)$.
  \item[vi.]  Denote by $\what{o}(t): C^m(M,E) \rightarrow C^r(M,E)$ ($0 \le r \le m$) a linear operator which has the following property: for any $\varphi \in C^m(N,E)$ and $q_0\in M$ there exists a neighbourhood $U$ such that
\be{os-d}
\lim_{t\to+0}\frac{\norm{\what{o}(t)(\varphi) (q)}}{t}=0
\ee
uniformly with respect to $q\in U$. Later we will develop a more detailed definition of such operators, as well as several useful examples.
\end{enumerate}

Of course, we can consider linear combinations of such   linear operators.

When $\varphi$ is a local diffeomorphism, these operators simply give local coordinate expressions.
We need not restrict ourselves to the space $C^m(M,E)$. Indeed, given any open set $U \subseteq M$, we may view $U$ as a Banach manifold
in its own right and therefore consider the space $C^m(U,E)$. For example, the local flow $P_{t_0,t} : J_0 \times U \rightarrow U$
of a vector field $V_t$ gives rise to a family of linear mappings $\what{P}_{t_0,t} : C^m(U,E) \rightarrow C^m(U,E)$.

Note that for operators $\what{P}$ the semigroup property \rref{sgp} for flow of diffeomorphism $P_{t_0,t}$
becomes
\be{sgp-o}
\what{P}_{t_0,s}\circ \what{P}_{s,t}=\what{P}_{t_0,t}.
\ee

An operator $\what{o}(t)$ \rref{os-d} will play an important role in the calculus of remainder terms  which will be developed later.
For an example of such operator $\what{o}(t)$ we consider a flow operator $\what{P}_{t}$ for an   autonomous vector field $V$. It follows from \rref{rie} that for the following operator
\be{os-e}
\what{o}(t):=\what{P}_{t} -\what{Id}_M-t\what{V}.
\ee
and a function $\varphi\in  C^m(M,E)$ we have that for any $q_0\in M$
\be{os-r}
\what{o}(t)(\varphi)(q)=\int_0^t (\varphi_*(P_s (q))V(P_s(q))-\varphi_*(q)V(q))\ d s
\ee
for all $q$ in a neighborhood of $q_0$.
  This representation implies that the operator  \rref{os-e}
satisfies \rref{os-d}.

\subsection{Differentiation and integration of operator-valued functions}

Consider an operator-valued function $t\to A_t$  whose values are linear mappings $A_t : C^m(M,E) \rightarrow C^p(M,E)$. This function is called {\em integrable} if for any $\varphi \in C^m(M,E)$ and $q\in M$ the function
$t\to A_t(\varphi) (q)$ is integrable.
Then the linear operator  $\left(\displaystyle \int_{t_0}^{t_1} A_\tau \, d \tau \right): \ C^m(M,E)\to C^{r}(M,E)$ is defined as follows
\[
 \left(\displaystyle \int_{t_0}^{t_1} A_\tau \, d \tau \right) (\varphi) (q) := \displaystyle \int_{t_0}^{t_1} A_\tau(\varphi)(q) \, d \tau
\]
It follows immediately from \rref{mde} and \rref{rie} that the flow operator
 $\what{P}_{t_0,t}$  representing
flow of diffeomorphisms for a nonautonomous vector field $V_t$ satisfies the integral
equation
\be{iem}
 \what{P}_{t_0,t} = \what{Id}_M+ \displaystyle \int_{t_0}^t \what{P}_{t_0,\tau} \circ \what{V}_\tau \, d \tau.
\ee
Moreover, we have that the unique operator valued solution of the integral
equation \rref{iem} is the function $t\to \what{P}_{t_0,t}$.

Now we introduce a concept of \emph{differentiability} of an operator-valued
function $A_t$.
The operator-valued function $A_t: C^m(M,E) \rightarrow C^r(M,E)$ is called
{\em differentiable} at $t$ if there exists a linear operator $B_t:
C^r(M,E) \rightarrow C^s(M,E)$
\be{dif-d}
A_{t+h}=A_t+ h\, B_t+\what{o}(h).
\ee
 The operator $\displaystyle \dfrac{d A_t}{dt}:=B_t$ is the \emph{derivative} of $A_t$.

This definition is well-suited for an operator $ \what{P}_{t_0,t}$ arising
from flow
diffeomorphisms representing differential equation \rref{mde} in the case
of the nonautonomous vector field $V_t$ which is continuous in $t$.
Namely, we have from the semigroup property  \rref{sgp-o} that
\[
\what{P}_{t_0,t+h}-\what{P}_{t_0,t}-h \what{P}_{t_0,t}\circ \what{V}_t=\what{P}_{t_0,t}\circ (\what{P}_{t,t+h}-\what{Id}_M-h  \what{V}_t).
\]
The last expression can be represented as
\be{ode-os}
\what{P}_{t_0,t}\circ \int_t^{t+h}(\what{P}_{t,s}\circ \what{V}_s -\what{V}_t)\ ds.
\ee
Using a representation for \rref{ode-os}  similar to the one from  \rref{os-r} and continuity $V_t$ in $t$, we obtain that  \rref{ode-os} is $\what{o}(h)$.

Thus, we have derived the representation
\be{fod}
\what{P}_{t_0,t+h}=\what{P}_{t_0,t}+h \what{P}_{t_0,t}\circ \what{V}_t+\what{o}(h).
\ee
This means that $t\to\what{P}_{t_0,t}$ is differentiable and
 for every $t$,
\begin{equation*}
\frac{d}{dt} \what{P}_{t_0,t}=\what{P}_{t_0,t}\circ \what{ V}_t.
\end{equation*}
We see that in this continuous in time $V_t$ case, the operator-valued function $\what{P}_{t_0,t}$ satisfies the differential equation
\be{dem}
\dfrac{d \what{P}_{t_0,t}}{dt} = \what{P}_{t_0,t} \circ \what{V}_t, \quad
\what{P}_{t_0,t_0}=\what{Id}_M.
\ee
It is easy to check that $\what{P}_{t_0,t}$ is the unique solution of
this operator differential equation and also of the operator integral equation \rref{iem}.

However, in the case when the vector field $V_t$ is only integrable in $t$  then a Carath\'eodory solution
$q(t)$ of the differential equation \rref{mde}  is an absolutely
continuous function and we cannot guarantee that $\what{P}_{t_0,t}$ is differentiable
for every $t$.

An operator-valued function $\what{A}_t$ is called \emph{absolutely continuous} on $[a,b]$ if $
\what{A}_t=\what{A}_{t_0}+\int_{t_0}^t \what{B}_\tau\, d\tau$ for any $t\in [a,b]$ for some integrable operator-valued function $\what{B}_t$. We denote $\what{B}_t$ as $\displaystyle \frac{d}{dt} \what{A}_t$ and understand this derivative in the sense of distributions\footnote{We use a term \emph{distribution}  by analogy with a  concept of a generalized derivative as a   distribution
 in the theory of linear partial differential operators (see \cite{horm}).}: for any
$t_1,t_2\in [a,b] $, for any $\varphi \in C^m(M,E)$ and $q\in M$
\[
\wha{t_2}(\varphi)(q) -\wha{t_1}(\varphi)(q)=\int_{t_1}^{t_2} \frac{d }{dt} \wha{t} \, (\varphi)(q) \, dt.
\]
\begin{remark}\label{acvf}
Let $W$ be a $C^m$ vector field and $\wha{t}$ is absolutely continuous then
$
\wha t\circ \what{W}
$
is also absolutely continuous and $\displaystyle \dop{t}(\wha{t} \circ W)=
\dop{t}\wha{t} \circ W$.
\end{remark}
Note that in the case when absolutely continuous operator-valued function $\what{A}_t$ is defined by a flow of diffeomorphisms $P_t:M\to M $ then for any
$q\in M$ the derivative $\displaystyle \dop{t} P_t(q)$ exists for a.a. $t\in [a,b]$.

As we demonstrated before, for measurable in $t$ vector fields $V_t$ the flow operator
$\what{P}_{t_0,t}$ is the unique absolutely continuous solution of the
integral operator equation \rref{iem}. In view of the definition of
the derivative of absolutely continuous operator-valued function, $\whp{t}$ is also unique solution of the operator differential equation
\rref{dem} in the sense of distributions.

\subsection{ Product Rules}

Here we discuss product rules for operator-valued functions
$\what{P}_t$ and $\what{Q}_t$.  We first establish such product rule
for the case when these functions are differentiable at $t$ in the sense of \rref{dif-d}, namely
\be{dif-pq}
\what{P}_{t+h}=\what{P}_t+h\frac{d}{dt}\what{P}_t  +\what{o}_1(h),  \quad
\what{Q}_{t+h}=\what{Q}_t+h\frac{d}{dt}\what{Q}_t   +\what{o}_2(h)
\ee
for some operators $\displaystyle \frac{d}{dt}\what{P}_t$ and $\displaystyle \frac{d}{dt}\what{Q}_t$.

\bt{pr1-t} Let operator-valued function $\what{P}_t$ and $\what{Q}_t$ be differentiable at $t$ and remainder terms $\what{o}_1$ and $\what{o}_2$ have the property
\be{rt1}
\what{o}_1(h)\circ\frac{d }{dt} \what{Q}_t+\frac{d }{dt}\what{P}_t\circ\what{o}_2(h)+ \what{o}_1(h)\circ\what{o}_2(h)=\what{o}(h).
\ee
Then operator-valued function $\what{P}_t\circ\what{Q}_t$ is differentiable at $t$ and
\be{pr1}
\frac{d}{dt}(\what{P}_t\circ\what{Q}_t)=\frac{d}{dt}\what{P}_t\circ\what{Q}_t+\what{P}_t\circ\frac{d}{dt}\what{Q}_t
\ee
\et
\begin{proof} It follows from \rref{dif-pq} and \rref{rt1} that
\[
\what{P}_{t+h}\circ\what{Q}_{t+h}=\what{P}_t\circ\what{Q}_t+h(\frac{d}{dt}\what{P}_t\circ\what{Q}_t+\what{P}_t\circ\frac{d}{dt}\what{Q}_t )+\what{o}(h)+\what{o}_1(h)\circ\what{Q}_t+\what{P}_t\circ\what{o}_2(h)
\]
But it is easy to see that the sum of last three terms is again operator $\what{o}(h)$ (see \rref{os-d}). This implies the differentiability of the product  $\what{P}_t\circ\what{Q}_t$ and the product rule \rref{pr1}.
\end{proof}

Thus, validity of a product rule in the form \rref{pr1} is reduced to the verification of the condition \rref{rt1}. We can verify directly that
\rref{rt1} holds
for flow operators $\what{P}_t$ and $\what{Q}_t$ which are operator solutions of the operator equation \rref{dem} or equation
\be{dem2}
\frac{d}{dt}\what{Q}_t=\what{W}_t\circ\what{Q}_t
\ee
for continuous in $t$ vector fields $V_t$ and $W_t$.

Now we consider a product rule in the sense of distributions  for  absolutely
continuous operator-valued functions $\what{P}_t$ and $\what{Q}_t$ which
 are represented for any $t\in(a,b)$ as
\be{pr-rep}
\whp t=\whp{t_0}+\int_{t_0}^t \frac{d}{d\tau}\what{P}_\tau\ d\tau, \quad
\whq t=\whq{t_0}+\int_{t_0}^t \frac{d}{d\tau}\what{Q}_\tau\ d\tau,
\ee
\begin{assumption}\label{pr-ass}\em
Let  $\whp t, \whq t$ be  absolutely continuous operator-valued functions
such that  for any $\varphi\in C^m(M,E)$ and $q \in M$ \\
(i) Function $\displaystyle t \to \whp{t} \circ \whq{t}(\varphi)(q)$ is continuous on $ (a,b)$.\\
(ii) Functions
\[
(t,\tau)\to  \dop{\tau} \whp{\tau}\circ \whq{t}(\varphi)(q),
\quad (t,\tau)\to \whp{t}\circ \dop{\tau} \whq{\tau}(\varphi)(q)
\]
are integrable on $(a,b)\times(a,b)$.\\
(iii) For any $t,t_1,t_2\in (a,b)$
\[
\int_{t_1}^{t_2}\dop{\tau} \whp{\tau}d\tau \circ \whq{t}(\varphi)(q)=
\int_{t_1}^{t_2}\dop{\tau} \whp{\tau}\circ \whq{t}(\varphi)(q)d\tau,
\]
\[
\whp{t}\circ \int_{t_1}^{t_2}\dop{\tau} \whq{\tau}d\tau (\varphi)(q)=
\int_{t_1}^{t_2} \whp{t}\circ \dop{\tau} \whq{\tau} (\varphi)(q) d\tau
\]
(iv) There exists an integrable function $k_1(\tau)$ such that for all small $h$,  all $t\in [\tau-h,\tau]$ and a.a. $\tau\in (a,b)$
\be{pr-ass2}
\norm{\dop{\tau} \whp{\tau}\circ \whq{t} (\varphi)(q)}\le k_1(\tau) \quad \norm{ \whp{t+h}\circ\dop{\tau} \whq{\tau} (\varphi)(q)}\le k_1(\tau)
\ee
\end{assumption}

Note that if $\whp t$,$\whq t$ are  absolutely continuous solutions
of \rref{dem} or \rref{dem2}, or they  are of the type presented in  Remark \ref{acvf} with $\wha t$ being a solution of \rref{dem} or \rref{dem2} then conditions {\em (i)-(iv)} are satisfied when $V_t$ and $W_t$ are locally integrable bounded.

\bt{pr-th} Let absolutely continuous operator-valued function
$\whp{t}$ and $\whq{t}$ satisfy Assumption \ref{pr-ass}. Then $\whp{t}\circ\whq{t}$ is absolutely continuous and for any
$t_1 , t_2$ in $(a,b)$

\be{pr-d}
\int_{t_1}^{t_2} \dop{t}( \whp{t}\circ\whq{t})\, dt= \int_{t_1}^{t_2}( \dop{t} \whp{t}\circ \whq{t}+ \whp{t} \circ \dop{t} \whq{t})\, dt.
\ee
\et
\begin{proof} Let us fix $t_1,t_2\in (a,b)$, $\varphi\in C^m(M,E)$ and $q \in M$ then
\be{pr-1}
\begin{split}
\int_{t_1}^{t_2} \frac{1}{h}(\whp{t+h}\circ\whq{t+h}-\whp{t}\circ\whq{t})(\varphi) (q) \, dt=\\= \frac{1}{h} \int_{t_2}^{t_2+h} \whp t \circ\whq{t}(\varphi) (q) \, dt -\frac{1}{h}\int_{t_1}^{t_1+h} \whp t\circ\whq{t}(\varphi) (q)  \, dt
\end{split}
\ee
Due to {\em(iii)} of Assumption \ref{pr-ass} we have
\be{pr-0}
\begin{split}
\int_{t_1}^{t_2} \frac{1}{h}(\whp{t+h}\circ\whq{t+h}-\whp{t}\circ\whq{t})   (\varphi)(q)\, dt=\\= \int_{t_1}^{t_2}dt
\frac{1}{h}\int_{t}^{t+h}\dop{\tau} \whp{\tau}\circ \whq{t}(\varphi)(q)\, d\tau
+\int_{t_1}^{t_2}dt \frac{1}{h}\int_{t}^{t+h}\whp{t+h}\circ \dop{\tau} \whq{\tau} (\varphi)(q) d\tau
\end{split}
\ee

By using Fubini theorem, we obtain  the following
\begin{lemma}\label{ft}
Let $g: (a,b)\times (a,b)\to E$ be integrable function then for any $t_1,t_2\in (a,b)$ and sufficiently small $h$
\be{ft-f}
\begin{split}
\int_{t_1}^{t_2} dt\int_{t}^{t+h} g(t,\tau) \,d\tau=\int_{t_1}^{t_2} d\tau\int_{\tau-h}^{\tau} g(t,\tau) \,dt-\\-\int_{t_1}^{t_1+h}d\tau\int_{\tau-h}^{t_1} g(t,\tau)\, d\tau+\int_{t_2}^{t_2+h}d\tau \int_{\tau-h}^{t_2} g(t,\tau)\,dt
\end{split}
\ee
\end{lemma}
We use this Lemma to evaluate the first term in the right-hand side of \rref{pr-0}
\[
\begin{split}
\int_{t_1}^{t_2}dt
\frac{1}{h}\int_{t}^{t+h}\dop{\tau} \whp{\tau}\circ \whq{t}(\varphi)(q)\, d\tau=\int_{t_1}^{t_2} d\tau \frac{1}{h}\int_{\tau-h}^{\tau}\dop{\tau} \whp{\tau}\circ \whq{t}(\varphi)(q) \,dt-\\-\frac{1}{h}\int_{t_1}^{t_1+h}d\tau\int_{\tau-h}^{t_1} \dop{\tau} \whp{\tau}\circ \whq{t}(\varphi)(q)\, d\tau+\frac{1}{h}\int_{t_2}^{t_2+h}d\tau \int_{\tau-h}^{t_2} \dop{\tau} \whp{\tau}\circ \whq{t}(\varphi)(q)\,dt
\end{split}
\]
It follows from conditions {\em (ii),(iv)} of Assumptions \ref{pr-ass}, from
Fubini theorem and Lebesgue convergence theorem that
\be{pr-lim1}
\lim_{h\to 0} \int_{t_1}^{t_2}dt
\frac{1}{h}\int_{t}^{t+h}\dop{\tau} \whp{\tau}\circ \whq{t}(\varphi)(q)\, d\tau=\int_{t_1}^{t_2}
\dop{\tau} \whp{\tau}\circ \whq{\tau}(\varphi)(q)\, d\tau
\ee
By similar argument we prove the following limit
\be{pr-lim2}
\lim_{h\to 0} \int_{t_1}^{t_2}dt
\frac{1}{h}\int_{t}^{t+h} \whp{t+h}\circ\dop{\tau} \whq{t}(\varphi)(q)\, d\tau=\int_{t_1}^{t_2}
 \whp{\tau}\circ \dop{\tau} \whq{\tau}(\varphi)(q)\, d\tau
\ee
By using these limits and continuity of $t\to \whp t \circ \whq t (\varphi)(q)$
(condition {\em (i)},
we derive from
 \rref{pr-1} and \rref{pr-0} that
\be{pr-2}
(\whp{t_2}\circ\whq{t_2}-\whp{t_1}\circ\whq{t_1})(\varphi)(q)=\int_{t_1}^{t_2} ( \whp{t}\circ  \dop{t} \whq{t}
+  \dop{t} \whp t  \circ \whq{t})\, dt (\varphi)(q)
\ee
This implies that $\whp{t}\circ   \whq{t}$ is absolutely continuous and its derivative satisfies the
product rule \rref{pr-d} in the sense of distributions.
\end{proof}

\subsection{Operators $\mathrm{Ad}$ and $\mathrm{ad}$.}

Let $V$ be a vector field and $F:M\to M$ be a $C^m$ diffeomorphism.
For a solution $q(t)$ of the equation $\displaystyle \dot q(t)=V(q(t))$ the function
$r(t)=F(q(t))$ is also a solution of the differential equation
\be{adde}
\dot r(t)=F_* (q(t))V(q(t))=F_*V(r(t))
\ee
where the vector field $F_*V$ is defined by $F_*V(r):= F_*(F^{-1}(r))V(F^{-1}(r))$.

To obtain the representation for the operator
$\what{F_* V}$ corresponding the vector field $F_* V$ we consider the diffeomorphism flow $R_t$ corresponding to the differential equation \rref{adde}. Then
\be{adde2}
\dop{t} \whr t=\whr{t}\circ \what{F_*V}
\ee
But $\whr t=  \whp t \circ \what{F}$ where $P_t$ is the diffeomorphism flow corresponding to the vector-field $V$. By using the product rule and \rref{adde2}
we get
\[
\dop t \whr t= \dop t \whp t \circ \what{F} =  \whp t \circ \whv{}\circ \what{F} =\whp t \circ \what{F} \circ \what{F_* V}
\]
This implies that
\be{ad-r}
\what{F_* V} = \what{F^{-1}}\circ \whv{} \circ \what{F}
\ee
 Following \cite{MR2062547} we define the
operator $\mathrm{Ad}\ \what{F} : \what{V} \mapsto \what{F} \circ \what{V}
 \circ \what{F^{-1}}$.

Recall that the Lie bracket $\left[V,W\right]$ of vector fields $V$ and $W$ is the vector field\footnote{To show that this is a vector-field we can use the relation \rref{com-d} for vector fields $V,W$.} whose operator representation has form  $\what{\left[V,W\right]} = \what{V} \circ \what{W} - \what{W} \circ\what{V}$.
Let us prove that the Lie bracket is invariant under diffeomorphism. We have
\begin{align*}
  \what{F_* \left[V,W\right]} & = \what{F^{-1}} \circ \left(\what{V} \circ \what{W} - \what{W} \circ \what{V} \right) \circ \what{F} \\
  & = \what{F^{-1}} \circ \what{V} \circ \what{F} \circ \what{F^{-1}} \circ \what{W} \circ \what{F} - \what{F^{-1}} \circ \what{W} \circ \what{F}
  \circ \what{F^{-1}} \circ \what{V}  \circ \what{F} \\
  & = \what{F_* V} \circ \what{F_* W} - \what{F_* W} \circ \what{F_* V}
  = \what{\left[F_* V,F_* W\right]}.
\end{align*}
Since the assignment $V \mapsto \what{V}$ is an injection, this proves  the vector field
equality $F_* \left[V,W\right] = \left[F_* V,F_* W\right]$.

It makes sense  (as in \cite{MR2062547})
  to define an operator $\mathrm{ad}\,\what{V}_t$ by
\be{ad2-d}
(\mathrm{ad}\,\what{V}_t)\circ \what{W}_t = \left[\what{V}_t,\what{W}\right].
\ee

Finally, let $v \in T_qM$ and $F : M \rightarrow M$ a diffeomorphism of class $C^m$. Then $F_* (q) v$ is a tangent vector in $T_{F(q)}M$ and it is natural
to ask for $\what{F_*(q) v}$ in terms of $\what{v}$ and $\what{F}$. We claim that, as in \cite{MR2062547}, one obtains
\be{pv}
\what{F_*(q) v} = \what{v} \circ \what{F}.
\ee
 To see this,
let $\varphi \in C^m(M,E)$. Then, applying the chain rule we have
$
  \what{F_* v}(\varphi) = \varphi_*(F(q)) F_*(q) v = \left(\varphi \circ F\right)_*(q) v
  = \what{v}\left(\varphi \circ F\right) = \what{v} \circ \what{F}(\varphi)
$.

\section{Operator Differential Equations and Their Applications}

In this section we further develop our extension in the direction of applications to flows of vector fields.

\subsection{Differential and integral operator equations}

Earlier, following \cite{MR524203,MR579930,MR2062547},    we have introduced the  operator differential equation
\be{ode1}
\dfrac{d }{dt} \what{P}_{t_0,t} = \what{P}_{t_0,t} \circ \what{V}_t, \quad \what{P}_{t_0,t_0}=\what{Id}
\ee
which has unique solution  $\what{P}_{t_0,t}$  representing
flow of diffeomorphisms for a nonautonomous vector field $V_t$
which is continuous in $t$.

In more general case of measurable in $t$ vector-field $V_t$   we have that  $\what{P}_{t_0,t}$ satisfies the integral operator equation
\be{oie1}
\what{P}_{t_0,t} = \what{Id}_M+ \displaystyle \int_{t_0}^t \what{P}_{t_0,\tau} \circ \what{V}_\tau \, d \tau
\ee
and it is the unique  absolutely continuous solution of this equation or the solution of the differential equation \rref{ode1} in sense of distributions.
The justification of this fact is based on the relation of $\what{P}_{t_0,t}$ to the Carath\'eodory solutions of \emph{ordinary} differential  equation \rref{mde}.

Now we consider the  differential operator equation
\be{ode2}
\dfrac{d}{dt} \what{Q}_{t_0,t}  =
- \what{V}_t \circ \what{Q}_{t_0,t}, \quad \what{Q}_{t_0,t_0}=\what{Id}_M.
\ee
Note that  this operator equation, even in the case $M=\Rb^n$, is related to some first-order
linear \emph{partial} differential equation.

The following result states that for
 a  locally integrable bounded $C^m$ vector field $V_t$  there exists a solution $\what{Q}_{t_0,t}$  of   \rref{ode2}
in the  sense of distributions. Moreover we have a representation of
$\what{Q}_{t_0,t}$  in terms of a solution of the equation of the type
\rref{oie1}.

 \begin{proposition}\label{inv-p}
 Let $V_t$ be a  locally integrable bounded $C^m$ vector field. Then absolutely
 continuous  operator-valued solutions  $\what{P}_{t_0,t}$ and
 $\what{Q}_{t_0,t}$ of differential equations \rref{ode1}
 and \rref{ode2} exist, are unique and
 \be{inv1}
 \what{Q}_{t_0,t}=(\what{P}_{t_0,t})^{-1}
 \ee
 \end{proposition}

\bpr Let $P_{t_0,t}$ be the flow of $V_t$, so that \rref{oie1} holds.
It is enough to prove the existence and uniqueness of \rref{ode2}.

Denote flow of diffeomorphisms $Q_{t_0,t}:=P_{t,t_0}$ then the operator-valued function $t\to \what{Q}_{t_0,t}$ is absolutely continuous and together with
$\what{P}_{t_0,t}$ satisfies Assumption \rref{pr-ass} for the product rule (Theorem
\rref{pr-th}).

Fix $\varphi\in C^m(M,E)$ and $q_0 \in M$. Then there exists an interval $(a,b)$
such that $\whp{t_0,t}(q)$ exists for any $t_0,t$ in $ (a,b)$ and any $q$
in some neighborhood of $q_0$. Due to the product rule we have for any $t\in(a,b)$
\[
\int_{t_0}^t  (\dop{\tau} \whp{t_0,\tau}\circ \whq{t_0,\tau}+ \whp{t_0,\tau} \circ \dop{\tau} \whq{t_0,\tau})\, dt (\varphi)(q_0)=0
\]
This implies that for a.a. $t\in (a,b)$
\[
( \whp{t_0,t}\circ \what{V}_t \circ \whq{t_0,t}+ \whp{t_0,t} \circ \dop{\tau} \whq{t_0,t}) (\varphi)(q_0)=0
\]
and $\what{Q}_{t_0,t}$ satisfies \rref{ode2} in the sense of distributions.

 To prove uniqueness of  such solution $\what{Q}_{t_0,t}$ we use  the product rule  \rref{pr-d}
\begin{align*}
  &  \ \what{P}_{t_0,t}\circ \what{Q}_{t_0,t} - \what{Id}_M = \int_{t_0}^t \frac{d}{d \tau} \left(  \what{P}_{t_0,\tau}\circ \what{Q}_{t_0,\tau}\right ) \, d \tau \\
  & = \int_{t_0}^t \left(\what{P}_{t_0,\tau} \circ \what{V}_\tau \circ \what{Q}_{t_0,\tau} - \what{P}_{t_0,\tau} \circ  \what{V}_\tau \circ \what{Q}_{t_0,\tau} \right) \, d \tau =0
\end{align*}
As a consequence, we have $\what{P}_{t_0,t} \circ \what{Q}_{t_0,t} = \what{Id}_M$ for all $t$, hence $\what{Q}_{t_0,t} = \what{P}_{t,t_0}$ which proves also
\rref{inv1}.
\epr

 \begin{proposition}\label{inv-p2}
 Let  $V_t$ be locally integrable bounded $C^m$ smooth vector field and $\whp{t_0,t}$ be an absolutely continuous solution of \rref{oie1}.
Then for any $C^m$ smooth vector field $W$ the operator-valued function
$t\to \mathrm{Ad}\, \what{P}_{t_0,t}\circ \what{W}$ is  absolutely continuous
 and satisfies the following
 equation in the sense of distributions
\be{ad-e}
\frac{d}{dt} \mathrm{Ad}\, \what{P}_{t_0,t}\circ \what{W} =
\mathrm{Ad}\, \what{P}_{t_0,t} \circ \mathrm{ad}\,\what{V}_t \circ \what{W}
\ee

 \end{proposition}

\bpr
Note that $\whp{t_0,t}^{-1}$ exists and due to   the assertion of Proposition \ref{inv-p} satisfies the differential equation \rref{ode2}. Then for any smooth vector field $W$
\begin{align*}
 \mathrm{Ad}\, \what{P}_{t_0,t}\circ\what{W}  =\what{W}+ \int_{t_0}^t \frac{d}{d\tau} \left(\what{P}_{t_0,\tau} \circ\what{ W} \circ \left(\what{P}_{t_0,\tau} \right)^{-1} \right) \, d\tau \\
   =
\what{Id}_M+ \int_{t_0}^t ( \what{P}_{t_0,\tau}\circ \what{V}_\tau \circ \what{W} \circ \left(\what{P}_{t_0,\tau} \right)^{-1} - \what{P}_{t_0,\tau} \circ \what{W} \circ \what{V}_\tau \circ \what{P}_{t_0,\tau}^{-1}) \, d\tau  \\
   = \what{Id}_M+ \int_{t_0}^t\left(\mathrm{Ad}\, \what{P}_t \right) \circ \what{\left[V_t,W\right]}\, d\tau.
\end{align*}
Recall the definition \rref{ad2-d} of the operator $\ad \what{V_t}$ to conclude
the proof.
\epr

   Method of variation of parameters can also be easy applied to the operator differential equation
\be{ode3}
\dfrac{d }{dt} \what{S}_{t_0,t} = \what{S}_{t_0,t} \circ (\what{V}_t
+\what{W}_t), \quad \what{S}_{t_0,t_0}=\what{Id}
\ee
Namely, we  have the following
\begin{proposition}\label{inv-p3}
 Let  $V_t$, $W_t$ be locally integrable bounded $C^m$ smooth vector fields.
Then
a solution of \rref{ode3} can be represented in the form
\be{mvp}
\what{S}_{t_0,t}=\what{C}_{t_0,t}\circ \what{P}_{t_0,t}
\ee
where  $\what{P}_{t_0,t}$ is the solution of the differential equation
\rref{ode1} and $\what{C}_{t_0,t}$ is a solution of the differential equation
\be{cf}
\dfrac{d }{dt} \what{C}_{t_0,t}= \what{C}_{t_0,t}\circ \mathrm{Ad}\, \what{P}_{t_0,t} \circ \what{W}_t, \quad \what{C}_{t_0,t_0}=\what{Id}_M
\ee
 \end{proposition}

\begin{proof} It follows from \rref{mvp} and Proposition \ref{inv-p} that
$\what{C}_{t_0,t}$ is absolutely continuous and by the product rule
\[
\begin{split}
\what{C}_{t_0,t}- \what{Id}_M=\int_{t_0}^t (
\dfrac{d }{d\tau}\what{S}_{t_0,\tau}\circ \what{P}^{-1}_{t_0,\tau}+\what{S}_{t_0,\tau}\circ  \dfrac{d }{d\tau} \what{P}^{-1}_{t_0,\tau}) \ d\tau=\\
 \int_{t_0}^t (
\what{C}_{t_0,\tau}\circ  \what{P}_{t_0,\tau}\circ (\what{V}_\tau+W_\tau)\circ \what{P}^{-1}_{t_0,\tau} - \what{C}_{t_0,\tau}\circ  \what{P}_{t_0,\tau}
\circ \what{V}_\tau\circ\what{P}^{-1}_{t_0,\tau})\ d\tau=\\
\int_{t_0}^t
\what{C}_{t_0,\tau}\circ  \what{P}_{t_0,\tau}\circ W_\tau\circ \what{P}^{-1}_{t_0,\tau}
\ d\tau  =  \int_{t_0}^t
\what{C}_{t_0,\tau}\circ \mathrm{Ad}\, \what{P}_{t_0,\tau}\circ \what{W}_\tau\ d\tau
\end{split}
\]
which proves \rref{cf}.
\end{proof}

\subsection{Derivatives of Flows with Respect to a Parameter}

Consider a family of nonautonomous $C^m$ vector field $V_t^\alpha$ which depends upon scalar
parameter $\alpha$ and corresponding flow $P^\alpha_{t_0,t}$,
Let us assume that  $V_t^\alpha$
is differentiable in $\alpha $ in the following sense
\[
V_t^\alpha = V_t + \alpha W_t + o_t(\alpha),
\]
where $V_t,W_t$ are nonautonomous $C^m$ vector fields which are locally bounded
(see Definition \ref{defn:locally-bounded-manifold}).

Let $\whp{t_0,t}$ and $\whq{t_0,t}$ be absolutely continuous solutions of \rref{ode1} and \rref{ode2}. We assume
that the  operator  $\what{o_t}(\alpha)$ is similar to the operator "little \emph{ o}" in \rref{os-d}
and satisfies
the following conditions
\[
\what{P}_{t_0,t} \circ \what{o_t}(\alpha) \circ \what{Q}_{t_0,t} = \what{o}(\alpha), \quad
\what{V}_t \circ \what{o_t}(\alpha) \circ \what{W}_t = \what{o}(\alpha)
\]
uniformly with respect to $t\in [t_0,t_1]$.

W use these assumptions  and  \eqref{cf} from Proposition \ref{inv-p3}  to obtain $\what{P}_{t_0,t}^\alpha = \what{C}_{t_0,t}^\alpha \circ \what{P}_{t_0,t}$, where $\what{C}_{t_0,t}^\alpha = \what{Id}_M + \alpha \int_{t_0}^t \mathrm{Ad} \, \what{P}_{t_0,\tau} \circ \what{W}_\tau\, d\tau + \what{o}(\alpha)$.
Hence $\what{P}_{t_0,t}^\alpha = \what{P}_{t_0,t} + \alpha \int_{t_0}^t \mathrm{Ad} \, \what{P}_{t_0,\tau} \circ \what{W}_\tau \, d\tau \circ \what{P}_{t_0.t} + \what{o}(\alpha)$.

This implies that  $\what{P}^\alpha_{t_0,t}$ is differentiable at $\alpha=0$
and
\begin{equation}
\label{eq:flow-der-first}
{\dfrac{\partial }{\partial \alpha}}\what{P}_{t_0,t}^\alpha = \int_{t_0}^t  \mathrm{Ad} \, \what{P}_{t_0,\tau} \circ \what{W}_\tau \, d\tau \circ \what{P}_{t_0,t} .
\end{equation}

This same formula is given in the context of the classical Chronological Calculus \cite{MR2062547}. We invite the reader to check that the second representation found in \cite{MR2062547}, given below
\begin{equation}
  \label{eq:flow-der-second}
  {\dfrac{\partial}{\partial \alpha}}\what{P}_{t_0,t}^\alpha = \what{P}_{t_0,t} \circ \int_{t_0}^t \mathrm{Ad} \what{P}_{t,\tau} \circ \what{W}_\tau \, d\tau
\end{equation}
is easily obtained from the first.

Operator formulas \rref{eq:flow-der-first} and \rref{eq:flow-der-second} can be used in order to obtain the following representations for derivative of the flow $P^\alpha_{t_0,t}$ at $\alpha=0$
\begin{equation}
\label{eq:out}
    \frac{\partial}{\partial \alpha} P_{t_0,t}^\alpha(q) = P_{t_0,t \, *} (q)\int_{t_0}^t P_{\tau, t_0 \, *}(P_{t_0, \tau}(q)) W_\tau(P_{t_0, \tau})(q) \, d \tau,
\end{equation}
\begin{equation}
\label{eq:in}
\frac{\partial}{\partial \alpha} P_{t_0,t}^\alpha(q) = \int_{t_0}^t P_{\tau, t \, *}(P_{t_0, \tau}(q)) W_\tau (P_{t_0, \tau}(q)) \, d \tau.
\end{equation}
Here we prove \rref{eq:in}, the proof of \rref{eq:out} is similar.
  By using \rref{eq:flow-der-second} and \rref{pv},  we obtain for any $\varphi\in C^M(M,E)$, the following obvious relations
\[
\begin{split}
\varphi_*(P_{t_0,t}(q))\frac{\partial }{\partial \alpha}P_{t_0,t}^\alpha  (q)=\what{q} \circ \dfrac{\partial }{\partial \alpha}\what{P}_{t_0,t}^\alpha (\varphi)=
\int_{t_0}^t \what{q}\circ \what{{P}}_{t_0,t} \circ \mathrm{Ad} \what{P}_{t,\tau} \circ \what{W}_\tau \, d \tau (\varphi)=\\
 \int_{t_0}^t \what{{P}_{t_0,t}(q)} \circ \what{P_{\tau, t \, *} W_\tau}(\varphi) \, d \tau= \int_{t_0}^t \varphi_*({P}_{t_0,t}(q)) {P_{\tau, t \, *}
({P}_{t_0,\tau}(q)) W_\tau}({P}_{t_0,\tau}(q)) \, d \tau
\end{split}
\]
These relations imply \rref{eq:in}.

\subsection{Finite sums of Volterra series and a remainder term estimate}

\comment{Let $V_t$ be a nonautonomous $C^m$ vector field on $M$ and let $P_{t_0,t} : J_0 \times U_0 \rightarrow U$ be the local flow of
of this field. As described above, we have $
  \what{P}_{t_0,t} = \what{Id}_M + \displaystyle \int_{t_0}^t \what{P}_{t_0,\tau}\circ \what{V}_\tau \, d \tau$.
Replacing $\what{P}_{t_0,\tau}$ with its integral form, we obtain
\begin{equation*}
  \what{P}_{t_0,t} = \what{Id}_M + \int_{t_0}^t \what{V}_\tau \, d \tau
  + \int_{t_0}^t \int_{t_0}^\tau \what{P}_{t_0,\sigma} \circ \what{V}_\sigma \circ \what{V}_\tau \, d \sigma d \tau.
\end{equation*}
Provided that natural $p \le m+1$, we may continue to obtain
\begin{equation}
\label{eq:volterra}
  \what{P}_{t_0,t} = \what{Id}_M + \sum_{i = 1}^{p-1} \int_{\Delta_i(t)}  \what{V}_{\tau_i} \circ \dots \circ
  \what{V}_{\tau_1} \, d \tau_i \dots d \tau_1 + \what{R}_p(t)
\end{equation}
where
\be{rt-d}
\what{R}_p(t):= \int_{\Delta_p(t)} \what{P}_{t_0,\tau_p} \circ \what{V}_{\tau_p} \circ \dots \circ
  \what{V}_{\tau_1} \, d \tau_p \dots d \tau_1,
\ee
is the remainder term and $\Delta_p(t)$ is the simplex $\left\{t_0 \le \tau_p \le \tau_{p-1} \le \dots \le \tau_1 \le t\right\}$.

Now consider the remainder term in \eqref{eq:volterra}. Fix $\varphi \in C^m(E,E)$. Choose $\delta$ small enough that all of the necessary
derivatives are bounded on $B(q_0;\delta)$. For $t$ sufficiently close to $t_0$, $P_{t_0,t}(q_0)$ will lie within $B(q_0;\delta)$. For
these values of $t$, we have
\begin{align*}
  & \left\| \int_{\Delta_r(t)} \what{P}_{t_0,\tau_r} \circ \what{V}_{\tau_r} \circ \dots \circ
  \what{V}_{\tau_1}(\varphi)(q_0) \, d \tau_r \dots d \tau_1 \right\| \\
  & \hspace{2cm} \le  \int_{\Delta_r(t)}  \left\| \what{P}_{t_0,\tau_r} \circ \what{V}_{\tau_r} \circ \dots \circ
  \what{V}_{\tau_1}(\varphi)(q_0)\right\| \, d \tau_r \dots d \tau_1 \\
  & \hspace{2cm} =  \int_{\Delta_r(t)}  \left\| \what{V}_{\tau_r} \circ \dots \circ
  \what{V}_{\tau_1}(\varphi)(P_{t_0,\tau_r}(q_0))\right\| \, d \tau_r \dots d \tau_1 \\
  & \hspace{2cm}\le \int_{\Delta_r(t)} \alpha(\tau_r) \dots \alpha(\tau_1)d \tau_r \dots d \tau_1 \left(\max_{i = 1, \dots, r} \sup_{q \in B(q_0;\delta)} \left\|\varphi^{(i)}(q) \right\|\right)\\
  &\hspace{2cm} = \frac{1}{r!} \left( \int_{t_0}^t \alpha(\tau) \, d \tau \right)^r \left(\max_{i = 1, \dots, r} \sup_{q \in B(q_0;\delta)} \left\|\varphi^{(i)}(q) \right\|\right).
\end{align*}
Therefore $\left\| \displaystyle \int_{\Delta_r(t)} \what{P}_{t_0,\tau_r} \circ \what{V}_{\tau_r} \circ \dots \circ
  \what{V}_{\tau_1}(\varphi)(q_0) \, d \tau_r \dots d \tau_1 \right\| = O(t^r)$ as $t \rightarrow 0$.
  In the case where $M$ is a general Banach manifold, we choose a coordinate neighborhood $U$ of $q_0$ and immediately obtain the same result
  in local coordinates. This proves
  \begin{proposition}
    Let $M$ be a $C^m$-manifold over a separable Banach space $E$ and let $V_t$ be a nonautonomous $C^m$-vector field on $M$. For any point $q_0 \in M$
    and any choice of local coordinates $\psi$ and any $\varphi \in C^m(M,E)$ at $q_0$, we have
    \begin{align*}
     & \left\| \what{P}_{t_0,t}(\varphi)(q_0) - \left( \what{Id}_M + \sum_{i = 1}^{r-1} \int_{\Delta_i(t)}  \what{V}_{\tau_i} \circ \dots \circ
  \what{V}_{\tau_1} \, d \tau_i \dots d \tau_1 \right)\left(\varphi\right)(q_0) \right\| \\
   & \hspace{7cm} \le \frac{1}{r!} \left(\int_{t_0}^t \alpha_{\psi,\varphi}(\tau) \, d\tau \right)^r = O(t^r).
    \end{align*}
    where $\alpha$ is an integrable function depending on $\varphi$ and $\psi$.
  \end{proposition}
}

Let $V_t$ be a nonautonomous $C^m$ vector field on $M$ and let $P_{t_0,t} : J_0 \times U_0 \rightarrow U$ be the local flow of
of this field. Consider the operator integral equation \rref{oie1}.
Replacing $\what{P}_{t_0,\tau}$ in \rref{oie1} with its integral form, we obtain
\begin{equation*}
  \what{P}_{t_0,t} = \what{Id}_M + \int_{t_0}^t \what{V}_\tau \, d \tau
  + \int_{t_0}^t \int_{t_0}^\tau \what{P}_{t_0,\sigma} \circ \what{V}_\sigma \circ \what{V}_\tau \, d \sigma d \tau.
\end{equation*}
Provided that $k \le m$, we may continue to obtain
\begin{equation}
\label{eq:volterra}
  \what{P}_{t_0,t} = \what{Id}_M + \sum_{i = 1}^{k-1} \int_{\Delta_i(t)}  \what{V}_{\tau_i} \circ \dots \circ
  \what{V}_{\tau_1} \, d \tau_i \dots d \tau_1 + \what{R}_k(t)
\end{equation}
where
\be{rt-d}
\what{R}_k(t):= \int_{\Delta_k(t)} \what{P}_{t_0,\tau_k} \circ \what{V}_{\tau_k} \circ \dots \circ
  \what{V}_{\tau_1} \, d \tau_k \dots d \tau_1,
\ee
is the remainder term and $\Delta_k(t)$ is the simplex $\left\{t_0 \le \tau_k \le \tau_{k-1} \le \dots \le \tau_1 \le t\right\}$.

Suppose that for any $\varphi \in C^m(M,E)$ and $q_0 \in M$ there is a neighborhood $U$ of $q_0$, $\delta>0$ and a constant $C$ such that for all $q \in U$, for any $t_0 \le \tau_k \le \dots \le \tau_1 \le t \le t_0 + \delta$, we have
\begin{equation}
\label{eq:bounded-remainder}
\left\|\what{V}_{\tau_k} \circ \dots \circ
  \what{V}_{\tau_1}(\varphi)(q) \right\|_E \le C
  \end{equation}
This is true, for example, when $V_t$ is locally bounded or autonomous.
Then
\begin{align}
  \left\|\what{R}_{k}(t)(\varphi)(q) \right\|_E & \le \int_{\Delta_k(t)} \left\| \what{V}_{\tau_k} \circ \dots \circ
  \what{V}_{\tau_1} (\varphi)(P_{t_0,\tau_k} (q)) \right\|_E\, d \tau_k \dots d \tau_1 \nonumber \\
  & \le \int_{\Delta_k(t)} C \, d \tau_k \dots d \tau_1 = \frac{C t^k}{k!}.\label{rt-est}
\end{align}
It follows for any $\varphi \in C^m(M,E)$ and any $q_0 \in M$, there is a neighborhood $U$ of $q_0$ on which the function $\displaystyle \frac{1}{t^{k-1}} R_{k}(t)(\varphi)$ converges uniformly to zero. The family $\what{R}_{k}(t)$ is an important example of a $\what{o}(t^{k-1})$ family of operators. In the following section, we rigorously define such families, establish their properties, and prove that $\what{R}_{k}(t) = \what{o}(t^{k-1})$.

\section{Calculus of little $o$'s}

We develop in this section the theory of operators of type $\what{o}(t^\ell)$
for $C^m$ smooth manifolds $M$.
We need the following definition

\begin{definition}\label{df-f}
  A set $\mathcal{F} \subset C^m (M,E)$ is called \emph{locally bounded} at  $q_0 \in M$ if there exists a coordinate chart $(\mathcal{O}, \psi)$ with $q_0 \in \mathcal{O}$ and a constant $C$ such that for any $i=0,\ldots,m$
\be{lo-c}
    \sup_{x \in \psi(\mathcal{O})} \left\|(\varphi\circ \psi^{-1})^{(i)}(x)\right\| \le C
\ee
\end{definition}

We say that a family of operators $\wha{t}: C^m(M,E)\to C^{m^\prime}(M,E)$, $t\in(-\delta,\delta)$ has a  defect $k_1:=\mbox{def } \wha{t}$ if for any $n$, $k_1\le n\le m$ and any $\varphi\in C^{\, n}(M,E)$ we have $\wha{t}\varphi\in C^{\hspace{1pt} n-k_1}(M,E)$.
A smooth vector field ${V_t}$ gives an example of the operator $\what{V_t}$ which has defect $1$.

\begin{definition}\label{lo-d}
A family of operators $\what{A}_t : C^m(M,E) \to C^{m^\prime}(M,E)$, $0<|t|<\delta$ with defect $k_1$ is called $\what{o}(t^k)$ if for any $q_0\in M$ and locally bounded at $q_0$ set $
\mathcal{F}\subset  C^m (M,E)$  there exists a coordinate chart $(\mathcal{O}, \psi)$ with
$q_0\in\mathcal{O}$ such that for any $i = 0,\ldots,m - k_1$
\be{lo-pd}
\lim_{t\to 0} \frac{1}{t^k}\norml{(\what{A}_t(\varphi)\circ \psi^{-1} (x))^{(i)}}=0
\ee
uniformly with respect to all $\varphi\in \mathcal{F}$, $x\in \psi(\mathcal{O})$. \end{definition}

The following proposition gives an important example of $\what{o}(t^k)$ operator.
\bp{rt-lo}
Let $C^m$ smooth vector field $V_t$ be locally bounded. Then
the remainder term operator $\what{R}_k(t)$ \rref{rt-d} is $\what{o}(t^{k-1})$ operator with the defect at most $k$.
\ep
\bpr Fix $q_0 \in M$ and a locally bounded at $q_0$ family of functions
$\mathcal{F}\subset C^m(M,E)$. Then there exists a constant $C$ such that
\rref{eq:bounded-remainder} holds for any $\varphi\in \mathcal{F}$. It implies that \rref{rt-est} holds for any $q\in \mathcal{O}$ where  $\mathcal{O}$ is some neighbourhood of $q_0$. This proves uniform convergence \rref{lo-pd}
for $i=0$. Similar argument demonstrates uniform convergence \rref{lo-pd}
for any $i=0,\ldots,m-k$.
\epr

Later we'll use the next properties of operators $\what{o}(t^k)$.

\bp{lop-pr}
Let $\what{o}(t^k)$ and $\what{o}(t^\ell)$  be families of operators
with defects $k_1$ and $\ell_1$ respectively. Then
 \begin{enumerate}[(i)]
   \item $\what{o}(t^k) + \what{o}(t^\ell) = \what{o}(t^{\min(k,\,\ell)})$  if $\max\{k_1,\ell_1\} \le m$;
    \item $\what{o}(t^k) \circ \what{o}(t^\ell) = \what{o}(t^{k+\ell})$
     if $k_1+\ell_1\le m$;
   \item $\displaystyle
     \what{X}_{t} \circ \what{o}(t^k) \circ \what{Y}_{t} = \what{o}(t^k)
     $ with a defect at most $k_1+2$ if $k_1 \le m-2$ and  $C^m$ smooth vector fields $X_t, Y_t$ are locally bounded;
   \item $\displaystyle
     \what{P}_{0,t} \circ \what{o}(t^k) \circ \what{Q}_{0,t} = \what{o}(t^k)$
     with a defect at most $k_1$
     if  $P_{0,t}$ and $Q_{0,t}$ are flows of locally bounded $C^m$ smooth vector fields;
   \item  $\what{P}_{0,t} = \what{Id}_E + \what{o}(1)$
   with the defect $0$ if $\what{P}_{0,t}$ is the flow of a locally bounded $C^m$ smooth vector field.
   \end{enumerate}
\ep

\bpr Let us fix $q_0\in M$, a locally bounded at $q_0$ family $\mathcal{F}\subset C^m(M,E)$
and a coordinate chart $(\mathcal{O},\psi)$.
We observe that for any $\varphi\in\mathcal{F}$,
$x\in \psi(\mathcal{O})$ and $i=0,\ldots,m-\max\{k_1,\ell_1 \}$
\[
\norml{((\what{o}(t^k)+\what{o}(t^\ell))(\varphi)\circ\psi^{-1}(x))^{(i)}}\le
\norml{(\what{o}(t^k)(\varphi)\circ\psi^{-1}(x))^{(i)}}+
\norml{(\what{o}(t^\ell)(\varphi)\circ\psi^{-1}(x))^{(i)}}
\]
This inequality implies { (i)}.

Note that the family of functions $\displaystyle \mathcal{B}:=\{\frac{1}{t^\ell}\what{o}(t^\ell)
(\varphi)
 \st \varphi\in \mathcal{F}, \ 0<|t|<\delta
 \}$ is locally bounded at $q_0$. Then { (ii)} follows immediately
 from the Definition \ref{lo-d}.

To prove (iii) we note that $\what{o}(t^k)\circ\what{Y}_t$  is $\what{o}(t^k)$
with the defect $k_1+1$. Then it is easy to see that $\what{X}_t\circ\what{o}(t^k)\circ\what{Y}_t$
is $\what{o}(t^k)$ with the defect $k_1+2$.
 The assertion (iv) follows from the obvious observation that
 $\what{o}(t^k)\circ\what{Q}_{0,t}=\what{o}(t^k)$ with the defect $k_1$ and from the uniform convergence  in \rref{lo-pd}.

The last assertion (v) follows from the integral representation \rref{oie1}
for $\what{P}_{0,t}$ and boundedness assumptions.
\epr

In the case of $C^\infty$  manifold $M$ and vector fields on it we
don't need
to use the concept of defect of operators which map $\varphi\in C^\infty(M,E)$
into $C^\infty(M,E)$.
\bd{ilb-d} For $C^\infty$  manifold $M$
 a set $\mathcal{F} \subset C^\infty (M,E)$ is called \emph{locally bounded} at  $q_0 \in M$ if for any natural $m$ there exists a coordinate chart $(\mathcal{O}, \psi)$ with $q_0 \in \mathcal{O}$ and a constant $C$ such that
\rref{lo-c} holds
for any $i=0,\ldots,m$.
\ed

\bd{ilo-d}
For $C^\infty$  manifold $M$
a family of operators $\wha{t}: C^\infty(M,E)\to C^{\infty}(M,E)$,
$t\in(-\delta,\delta)$,
 is called $\what{o}(t^k)$ if for any $q_0 \in M$,  for any locally bounded at $q_0$
 set  $\mathcal{F}\subset C^\infty (M,E)$ and for any integer $m$ there exists  a coordinate chart $(\mathcal{O}, \psi)$ with $q_0 \in \mathcal{O}$
  such that for any $i=0,\ldots, m$ the limit \rref{lo-pd} takes place
uniformly with respect to all $\varphi \in \mathcal{F}$, $x\in \psi(\mathcal{O})$.
\ed
Then we have
\bp{p4.3}
For $C^\infty$  manifold $M$ and  locally
bounded  $C^\infty$  vector fields $X_t$ and $Y_t$ assertions (i)-(v) of Proposition
\ref{lop-pr} hold with $k_1=0$, $\ell_1=0$ and $m=\infty$.
\ep

\section{Commutators of flows and vector fields}

 Let $P_t$ and $Q_t$ be flows on a $C^m$  manifold $M$, generated by $C^m$  vector fields $X$ and $Y$, so that $\what{P}_0=\what{Id}$, $\what{Q}_0=\what{Id}$,
\begin{equation*}
  \frac{d \what{P}_t}{dt} = \what{P}_t \circ \what{X}, \hspace{1cm}  \frac{d \what{Q}_t}{dt} = \what{Q}_t \circ \what{Y}.
\end{equation*}
Following \cite{MR1222291}, we define a \emph{bracket of flows} $\left[P_t, Q_t\right] = Q_t^{-1} \circ P_t^{-1} \circ Q_t \circ P_t$
and we note that
\begin{equation*}
  \what{\left[P_t, Q_t\right]} = \what{P}_t \circ \what{Q}_t \circ \what{P}_t^{-1} \circ \what{Q}_t^{-1}.
\end{equation*}
In the case of finite dimensional manifolds, it follows from  the classical result that
\be{com-d}
   \what{\left[P_t, Q_t\right]} = \what{Id} + t^2 \left[\what{X}, \what{Y}\right] + \what{o}(t^2).
\ee
In the case of infinite dimensional manifolds and flows $P^i_t$, $i = 1, \dots, k$, generated by vector fields $X_i$, the general formula for an arbitrary bracket expression $B\left(P_t^1, \dots, P_t^k\right)$ was proved by Mauhart and Michor in \cite{MR1222291}. In operator notation, the general formula is
\begin{equation}
\label{eq:bracket-expression}
  \what{B\left(P_t^1, \dots, P_t^k\right)} = \what{Id} + t^k B\left(\what{X}_1, \dots, \what{X}_k\right) + \what{o}(t^k).
\end{equation}
Here we use the Chronological Calculus to prove this formula. In particular, we will establish the following:
\begin{theorem}[Mauhart and Michor]
\label{thm:mm}
  Let $M$ be an $C^m$ Banach manifold and $X_1, \dots, X_k$, $k\le m$, be $C^m$ smooth
  vector fields. Then for any bracket expression $B\left(P_t^k, \dots, P_t^1\right)$ we have the presentation \eqref{eq:bracket-expression} where $\what{o}(t^{k})$ has  defect at most $k-1$.
\end{theorem}
The advantage of our approach follows from the fact that the main part of the proof is reduced to algebraic computations. Moreover, an algorithm for deriving
a representation for remainder term in \rref{eq:bracket-expression} is given.

We will need the following results for families of local diffeomorphisms of the form
\begin{equation}
\label{eq:form-of-PQ}
  \what{P}_t = \what{Id} + t^m \what{X} + \what{o}(t^m) \hspace{1cm} \what{Q}_t = \what{Id} + t^n \what{Y} + \what{o}(t^n).
\end{equation}

\begin{proposition}
Let ${X}$ be a $C^m$  vector field. Then
\begin{equation}
\label{eq:sign-on-inverse}
  \what{P}_t^{-1} = \what{Id} - t^m \what{X} + \what{o}(t^m).
\end{equation}
\end{proposition}

\begin{proof}
  Consider flows $S_t$ and $T_t$ defined by $\dfrac{d \what{S}_t}{dt} = \what{S}_t \circ \what{X}$, $\what{S}_0 = \what{Id}$, and $\dfrac{d \what{T}_t}{dt} = - \what{X} \circ \what{T}_t$, $\what{T}_0 = \what{Id}$. Then $\what{T}_t = \what{S}_t^{-1}$, $\what{S}_t = \what{Id} + t \what{X} + \what{o}(t)$, and $\what{T}_t = \what{Id} - t \what{X} + \what{o}(t)$. In particular,
$ \what{P}_t = \what{S}_{t^m} + \what{o}(t^m)$.

  By applying $\what{P}_t^{-1}$ from the right, we get
$
    \what{Id} = \what{S}_{t^m} \circ \what{P}_t^{-1} + \what{o}(t^m)$.
Then by   applying $\what{T}_{t^m}$ from the left, we get
$
    \what{T}_{t^m} = \what{P}_t^{-1} + \what{o}(t^m)$,
  and so $\what{P}_t^{-1} = \what{Id} - t^m \what{X} + \what{o}(t^m)$.
\end{proof}

\begin{proposition}
\label{prop:one-bracket}
  Let families of local diffeomorphisms $P_t$ and $Q_t$ satisfy \eqref{eq:form-of-PQ}. Then
  \begin{equation}
  \label{eq:one-bracket}
    \what{\left[P_t, Q_t \right]} = \what{Id} + t^{m+n} \left[\what{X}, \what{Y} \right] + \what{o}\left(t^{m+n}\right).
  \end{equation}
\end{proposition}
\begin{proof}
  Recall that $\what{\left[P_t, Q_t\right]} :=  \what{P}_t \circ  \what{Q}_t  \circ\what{P}_t^{-1}\circ \what{Q}_t^{-1}$. Write
\[
    \what{P}_t^{-1} = \what{Id}_M - \what{V}_1, \hspace{.5cm} \what{P}_t = \what{Id}_M + \what{V}_2, \hspace{.5cm} \what{Q}_t^{-1} = \what{Id}_M - \what{W}_1, \hspace{.5cm} \what{Q}_t = \what{Id}_M + \what{W}_2.
\]
  Then
  \begin{equation}
  \label{eq:vw-brackets}
    \what{\left[P_t,Q_t\right]} =\left(\what{Id}_M + \what{V}_2\right) \circ\what{Q}_t \circ  \left(\what{Id}_M - \what{V}_1\right) \circ  \what{Q}_t^{-1} .
  \end{equation}
  Now,
$
    \what{Q}_t \circ \left(\what{Id}_M - \what{V}_1\right) \circ \what{Q}_t^{-1}  = \what{Id}_M - \what{Q}_t \circ \what{V}_1 \circ \what{Q}_t^{-1}
     = \what{Id}_M - \left(\what{Id}_M + \what{W}_2\right) \circ \what{V}_1 \circ \left( \what{Id}_M - \what{W}_1\right)
     = \what{Id}_M - \what{V}_1 -  \what{W}_2 \circ \what{V}_1 + \what{V}_1\circ \what{W}_1  + \what{W}_2 \circ \what{V}_1 \circ \what{W}_1$.

  Substituting this expression into \eqref{eq:vw-brackets} gives
  \begin{align*}
&   \what{ \left[P_t, Q_t\right] } =  \left(\what{Id}_M + \what{V}_2\right)\circ \left(\what{Id}_M - \what{V}_1 -  \what{W}_2 \circ \what{V}_1 + \what{V}_1\circ \what{W}_1  + \what{W}_2 \circ \what{V}_1 \circ \what{W}_1\right)\\
    & = \what{Id}_M - \what{V}_1 -  \what{W}_2 \circ \what{V}_1 + \what{V}_1\circ \what{W}_1  + \what{W}_2 \circ \what{V}_1 \circ \what{W}_1  +\what{V}_2 -\what{V}_2\circ \what{V}_1\\ & - \what{V}_2\circ \what{W}_2 \circ \what{V}_1 + \what{V}_2\circ\what{V}_1\circ \what{W}_1
      + \what{V}_2\circ \what{W}_2 \circ \what{V}_1 \circ \what{W}_1
  \end{align*}
But
\[
  - \what{V}_1 + \what{V}_2 - \what{V}_2 \circ \what{V}_1 = \what{P}_t^{-1} - \what{Id}_M + \what{P}_t - \what{Id}_M - \left(\what{P}_t - \what{Id}_M\right) \circ \left( \what{Id}_M - \what{P}_t^{-1} \right) = 0.
\]

Therefore,
\[
  \what{ \left[P_t, Q_t\right] } = \what{Id}_M -  \what{W}_2 \circ \what{V}_1 + \what{V}_1\circ \what{W}_1  + \what{R}
\]
  where
\be{mmf-r}
 \what{R}:= \what{W}_2 \circ \what{V}_1 \circ \what{W}_1  - \what{V}_2\circ \what{W}_2 \circ \what{V}_1 + \what{V}_2\circ\what{V}_1\circ \what{W}_1
      + \what{V}_2\circ \what{W}_2 \circ \what{V}_1 \circ \what{W}_1
\ee
 By \eqref{eq:form-of-PQ} and \eqref{eq:sign-on-inverse} we have
$\what{V}_1 = t^m \what{X} + \what{o}(t^m)$, $ \what{V}_2 = t^m \what{X} + \what{o}(t^m)$,
$\what{W}_1 = t^n \what{Y} + \what{o}(t^n)$, $\what{W}_2 = t^n \what{Y} + \what{o}(t^n)$.

By using Proposition \ref{prop:one-bracket} we obtain from previous relations and \rref{mmf-r} that $
 \what{R}=\what{o} (t^{m+n})$ and
\[
\what{V}_1\circ \what{W}_1 - \what{W}_2 \circ \what{V}_1 = t^{m+n} \left[ \what{X}, \what{Y} \right] + \what{o}(t^{m+n}).
\]

This proves \eqref{eq:one-bracket}.
\end{proof}
Applying \ref{prop:one-bracket} inductively, we obtain Theorem \ref{thm:mm}. Note that in the process of proving the theorem, we have obtained an expression for the remainder term.

\section{Chow-Rashevskii theorem for infinite-dimensional manifolds}

Consider an $n$-dimensional manifold $M$ with a sub-riemannian
{\em distribution}
$\mathcal{H} \subset TM$, which, by definition, is a vector sub-bundle of
the tangent bundle $TM$ of the manifold with an inner product on its fiber space \cite{ mon02}.
An absolutely continuous curve
$q: [0,T] \to M$ is called
{\it horizontal}  if its derivative belongs to
 $\mathcal H$
for almost all $t$.

The classical Chow-Rashevskii theorem \cite{chow,rash}
provides conditions in terms of basis vector fields $\left\{V_i\right\}_{i=1,\ldots,m}$ of
the distribution $\Hc$ and their iterated Lie brackets for connectivity of
arbitrary
two points of the sub-riemannian manifold by a horizontal curve.

Namely, let us consider the following distribution $\Lc$
which is defined point-wise as the linear span of the set generated by iterated
Lie brackets of basis vector fields $\left\{V_i\right\}_{i=1,\ldots,m}$ as follows:
\be{lie}
\Lc [V_1,\ldots,V_m](q):=\lspan \{
B(V_\ino,V_\ind,\ldots, V_{i_{k-1}},V_{i_k})(q) \st
 k=1,2,\ldots \}
\ee
The classical Chow-Rashevskii theorem states that the condition
\be{crc}
\Lc [
V_1,\ldots,V_m ](q) =TM (q) \quad \forall \ q\in M
\ee
implies the connectivity of any two points on the manifold $M$ by a horizontal
curve.

Historically this theorem has played a fundamental role in nonlinear
control theory \cite{Isi95,jur,bloch2003,MR2302744} by demonstrating that the condition \rref{crc} is a sufficient
 for the global controllability
of the following affine-control system:
\be{acs}
\dot q=\sum_{i=1}^m u_i (t) V_i (q).
\ee

Here we are interested in generalizing these sufficient conditions
for  global controllability for the case of infinite-dimensional
manifold $M$. Consider an affine control system
\be{acsi}
\dot q=\sum_{i=1}^\infty u_i(t) V_i (q)
\ee
where $V_i$ are smooth vector-fields on $M$, and $u(t):=(u_1(t),u_2(t),\ldots) $ is a control.

Let $M$ be an infinite-dimensional $C^\infty$ smooth connected  manifold
   \cite{krieg} with underlying  smooth Banach space
$E$. The concept of smooth Banach space will be discussed in the next subsection.

 A control $u(t)$ is called \emph{admissible} if it is piecewise constant
and at each $t$ only a finite number of its components $u_i(t)$ are different from zero and take values $+1$ or $-1$. The set of all admissible controls
is denoted $\Uc$.

Note that for any initial point $q_0$ for any admissible control $u(t)$
there exists (at least locally) a unique solution $q(t;q_0,u)$ of the control
system \rref{acsi}. This solution we call a {\em trajectory}.
A {\em reachability set} for the initial point $q_0$
\be{set}
\Rc (q_0):=\{ q(t;u,q_0) \st  \forall \ t\ge 0, \ \forall \ u \in \Uc \}
\ee
consists of all points of all trajectories of
 \rref{acsi} corresponding to all admissible controls
 $u \in \Uc$. Thus, the set $\Rc(q_0)$ consists of all points to which the control system can be driven from the point $q_0$ using admissible controls.

 Here we provide infinitesimal conditions  in terms
of vector fields $\left\{V_i \right\}_{i=1,2,\ldots}$, their Lie brackets and bracket iterations similar to \rref{crc} which imply global approximate controllability
of the system \rref{acsi}.

\bd{d1}
Control system \rref{acsi} is called global approximate controllable
if for any
 $q_0 \in M$
\be{ac-d}
\overline{\Rc (q_0)}=M
\ee
\ed

Thus, global approximate controllability of system
 \rref{acsi} means that for arbitrary points
$q_0, q_1 \in M$ and any open neighbourhood  $\mathcal O$
of the point $q_1$
there exists an admissible control $u \in \Uc$ such that
at  some moment $T$ the trajectory
 $x(T;u,x_0)$ enters the neighbourhood   $\mathcal O$.

 For the family of smooth vector fields $V_i$, $i=1,2,\ldots$  define
the following set similar to \rref{lie}
\be{liei}
\Lc [V_1,V_2,\ldots](q):=\lspan \{
B(V_\ino,V_\ind,\ldots, V_{i_{k-1}},V_{i_k})(q) \st
 k=1,2,\ldots \}
\ee

Note that in the definition \rref{liei} of $\Lc$ we consider only brackets $B$ which are well defined.

\bt{th1} { Let $M$ be an infinite-dimensional smooth
manifold associated with the smooth Banach space $E$ and a
smooth affine-control system
\rref{acsi} satisfies
\be{crci}
\overline{\Lc [V_1,V_2,\ldots](q)}=T_q M \quad \forall q\in M
\ee
Then system
\rref{acsi} is globally approximate controllable.
}
\et

The proof of this variant of Chow-Rashevskii theorem for infinite-dimensional
manifolds is based on the use of some constructions of nonsmooth analysis \cite{clswbook1,schiro}
and a characterization of the property of strong invariance \cite{clswbook1} of sets with respect to solutions of the control system \rref{acsi} and is similar to the proof of the analogous result for the case of Hilbert space $E$ \cite{led04}.

\vspace{15pt}

\subsection{Nonsmooth analysis on smooth manifolds and strong invariance of sets}

Concepts of \emph{strong} and \emph{weak invariance} play important role
in control theory (see \cite{clswbook1} for finite-dimensional results
 and \cite{clr} for related results on approximate invariance in Hilbert spaces).
A set $S\subset M$ is called \emph{strongly invariant} with respect to
trajectories of a control system \rref{acsi} if for any $q_0\in S$ and any
admissible control $u\in \Uc$ the trajectory $q(t;q_0,u)$ stays in $S$ for
all $t>0$ sufficiently small.
Note that the fact that  reachability set $\Rc(q_0)$ \rref{set} is strongly invariant follows immediately from its definition.

Here we provide infinitesimal conditions for strong invariance of a closed set $S$ in terms of normal vectors to
$S$ and iterated Lie brackets of vector fields $V_i$, $i=1,2,\ldots$.
In order to define such normal vectors we need to recall some facts from
nonsmooth analysis on smooth Banach spaces and on smooth infinite-dimensional
manifolds.

A Banach space $E$ is called \emph{smooth}  if there exists a non-trivial Lipschitz
$C^1$-smooth bump function (that is, a function with a bounded support). For example, Banach spaces with differentiable norm are smooth Banach spaces as,
in particular, Hilbert spaces are.

A subgradient $\zeta\in E^*$ of function $f: E\to (-\infty,+\infty]$ at the point $x$ is defined as follows: let there exist a $C^1$-smooth function $g:E\to\Rb$
such that the function $f-g$ attains a local minimum at $x$ then the subgradient
of $f$ at $x$ is the vector $\zeta=g' (x)$. The set of all subgradients at $x$ is called a \emph{subdifferential} $\partial_F f(x)$. It can be shown that
for lower semicontinuous functions $f$ subdifferentials are nonempty on a set which is dense in the domain of $f$. The detailed calculus of such
subdifferentials can be found in the monographs \cite{borzhu,clswbook1,schiro}.
The monograph \cite{clswbook1} is dedicated to the calculus of proximal subgradients in Hilbert spaces.

We also need the following mean-value inequality for lower semicontinuous function $f$: for any $r, x,y$ such that $r<f(y)-f(x)$ and for any $\delta>0$ there
exists a point $z\in [x,y]+\delta B$ and $\zeta\in\partial_F (z)$ such that
\be{mvi}
r< \ip{\zeta}{y-x}
\ee
(see \cite{clswbook1} for the original Hilbert space case and \cite{borzhu} for general smooth Banach space case).

The nonsmooth analysis for nonsmooth  semicontinuous functions on smooth finite-dimensional
manifolds was suggested in \cite{ledzhu07}. But the concept of subgradient of lower semicontinuous function from \cite{ledzhu07} is easily adapted for infinite-dimensional manifolds: $\zeta \in  T^{*}_q M$ is a subgradient of $f:M\to (-\infty,+\infty]$ if there exists a locally $C^1$-smooth function $g:M\in\Rb$ such that $f-g$ attains its local minimum at $q$ and $\zeta=d g(q)$.

Let $S\subset M$ be a closed subset on $M$ then the characteristic function
\[
\chi_S (q)=0, \quad q\in S, \quad \quad
\chi_S (q)=+\infty, \quad q\not\in S,
\]
is lower semicontinuous function on $M$.
\bd{nv} Vector $\zeta\in T_q^{*}M$ is called a normal vector
to a set $S$ at $q$ if $\zeta\in\partial_F \chi_S (q)$.
\ed
The set of all normal vectors is a cone and it is called a normal cone
$N_q S$
to the $S$ at $q$.

\bp{nvnz} Let $q'\in S$ be a boundary point of the closed set $S$ then
 any neighbourhood $\Oc$ of $q'$ contains a point $q\in S$ such that
 there exists a normal vector $\zeta\neq 0$, $\zeta\in N_q S$.
\ep

The proof of this proposition follows immediately from the mean-value inequality
\rref{mvi} and is left to a reader.

\bt{csit} The closed set $S\subset M$ is strongly invariant
with respect to solutions of the control system \rref{acsi} if
and only if
\be{csi}
\ip{\zeta}{B(V_\ino,V_\ind,\ldots, V_{i_{k-1}},V_{i_k})(q)}=0
\ee
for any  for any iterated Lie bracket of vector fields $V_i$, any normal vector
$\zeta\in N_q S$ and any $q\in S$.
\et

\bpr
 Let us assume that the set $S$ is strongly invariant, $q\in S$ and $\zeta=d g(q)\in N_q S$. This implies that
\be{ch1}
\chi_S(q') -\chi_S(q) \ge g(q')-g(q)
\ee
for some smooth function $g$ and all $q'$ near $q$.

Let us fix some iterated Lie bracket  $B(V_\ino,V_\ind,\ldots, V_{i_{k-1}},V_{i_k})(q)$ as in \rref{csi}, denote it $v$
and relate to it
an appropriate iterated flow bracket as in Section 5. For arbitrarily small $t>0$ we can find an
admissible control $u\in\Uc$ associated with this iterated flow
bracket such that we have in accordance with Theorem \rref{thm:mm}
\[
g(q(t;q,u))=g(q)+t^k \ip{d g(q)}{v}+o(t^k)
\]
Then we obtain from \rref{ch1} that
\[
t^k \ip{d g(q)}{v}+o(t^k)\le 0
\]
Of course, we can easily derive  \rref{csi} from this inequality.

To prove that conditions \rref{csi} imply strong invariance of the set
$S$ with respect to solutions of the affine control system  \rref{acsi}, we
can use methods of \cite{clr} for the characterization of strong and weak
approximate invariance of sets with respect to solutions of differential
inclusions.
\epr

\subsection{Proof  of an infinite-dimensional variant of  Chow-Rashevskii theorem}

Consider the reachability set $\Rc(q_0)$ \rref{set} and recall that this set is strongly invariant.
Note that for a fixed admissible control $u$, the function $q\to q(t;q,u)$ is continuous. This implies the important fact that the closure of the reachability set $\overline{\Rc (q_0)}$ is also strongly invariant.

Now let us assume that Theorem \ref{th1} is not true and $ \overline{\Rc (q_0)}\neq M$ for some $q_0\in M$. This implies the existence of some border point $q'$ of
$\overline{\Rc (q_0)}$.
Due to Proposition \ref{nvnz} there exists a point $q\in \overline{\Rc (q_0)}$
and a nonzero normal vector $\zeta$  at $q$ to it.
But due to strong invariance of $\overline{\Rc (q_0)}$ we have for the normal vector $\zeta$ that
 \rref{csi} holds for any iterated Lie bracket. In view of condition
 \rref{crci} it implies that $\zeta=0$ and this contradiction proves Theorem
 \ref{th1}.

\comment{ 

}

\end{document}